\documentclass[11pt,leqno]{article}

\usepackage[active]{srcltx}

\usepackage{amsfonts,latexsym,amsmath,amssymb,amsthm}
\usepackage{dsfont}
\usepackage{fullpage}
\newtheorem{theorem}{Theorem}[section]
\newtheorem{lemma}[theorem]{Lemma}

\newtheorem{proposition}[theorem]{Proposition}

\newtheorem{remark}[theorem]{Remark}

\theoremstyle{definition}

\theoremstyle{remark}
\newtheorem*{note*}{Note}

\numberwithin{equation}{section}

\makeatletter

\newcommand{\rank}{\mathop{\operator@font rank}}
\newcommand{\conv}{\mathop{\operator@font conv}}
\newcommand{\vol}{\mathrm{vol}}

\newcommand{\onetagright}{\tagsleft@false}
\makeatother

\makeatletter
\newcommand{\ls}{\leqslant}
\newcommand{\gr}{\geqslant}
\makeatother

\usepackage{color}

\usepackage{hyperref}

\begin{document}
\small

\title{\bf Regular functional covering numbers}

\author{Apostolos Giannopoulos and Natalia Tziotziou}

\date{}

\maketitle

\begin{abstract}
\footnotesize We establish the existence of a \emph{regular} functional $M$-position, in the sense of Pisier, for geometric log-concave functions. This provides a functional analogue of Pisier's regular $M$-positions for convex bodies and yields uniform control of covering numbers at all scales. Specifically, we show that every isotropic geometric log-concave function $f:\mathbb{R}^n \to [0,\infty)$ satisfies, for all $t\gr 1$,
$$\max\Bigl\{N(f,t\odot g),N(g,t\odot f^{\ast})\Bigr\}\ls\exp\!\left(\gamma_{n} n/t\right)\quad\text{and}\quad  
\max\Bigl\{N(f^{\ast},t\odot g),N(g,t\odot f)\Bigr\}\ls\exp\!\left(\delta_{n} n/t\right),$$
where $f^{\ast}$ denotes the Legendre dual of $f$, $(t\odot f)(x)=f(x/t)$ is the $t$-homothety of $f$, and $\gamma_n \ls c(\ln n)^2$, $\delta_n\ls c\ln n$. Our result shows that the isotropic position of a log-concave function already provides an almost $1$-regular functional $M$-position.
\end{abstract}

\section{Introduction}\label{section-1}

The study of covering numbers lies at the intersection of asymptotic geometric analysis and
high-dimensional probability. Sharp covering estimates have found important applications
in analysis, geometry, probability and combinatorics. Milman's theory of $M$-positions reveals 
that every convex body has a highly regular affine image whose covering behavior, and that of its polar, exhibit
near-optimal exponential bounds. Recent developments have extended these ideas to log-concave functions,
uncovering a rich functional counterpart to classical convex-geometric notions.

The purpose of this paper is to show that geometric log-concave functions
admit a regular functional $M$-position, in the sense of Pisier, and that, remarkably, 
the isotropic position already provides such a regular position. In particular, isotropic log-concave 
functions satisfy almost $1$-regular covering estimates at all scales.

\smallskip

Let $K$ and $T$ be convex bodies in $\mathbb{R}^{n}$. The covering number $N(K,T)$ is the smallest number of translates of $T$
needed to cover $K$:
$$N(K,T)= \min\Bigl\{ N\in\mathbb{N} : \exists\, x_1, \ldots, x_N \in \mathbb{R}^n
\ \text{such that} \ K \subseteq \bigcup_{j=1}^{N} (x_{j}+T)\Bigr\}.$$
A classical theorem of V.~Milman~\cite{VMilman-1986} asserts that every
centered convex body $K$ can be placed in $M$-position, namely there
exists a linear image $\widetilde K$ with $\vol_{n}(\widetilde K)=\vol_{n}(B_{2}^{n})$
such that
\begin{equation}  \label{betaMposition}
\max\Bigl\{N(\widetilde{K}, B_2^n),\ N(B_2^n, \widetilde{K}),\
N(\widetilde{K}^{\circ}, B_2^n),\ N(B_2^n, \widetilde{K}^{\circ})\Bigr\}\;\ls\; \exp(\beta n),
\end{equation}
for an absolute constant $\beta>0$, where $B_2^n$ is the Euclidean unit ball and 
$\widetilde{K}^{\circ}$ is the polar body of $\widetilde{K}$. Background material on convex bodies
is collected in Section~\ref{section:2}.

\smallskip

The framework of covering numbers was extended to functions by Artstein-Avidan and Slomka in
\cite{Artstein-Slomka-2015} (see also the earlier work \cite{Artstein-Raz-2011} of Artstein-Avidan
and Raz). Given measurable $f,g:\mathbb{R}^{n}\to[0,\infty)$, the functional covering number
of $f$ by $g$ is defined by
$$N(f,g)= \inf\Bigl\{ \mu(\mathbb{R}^{n}) : \mu\ast g \gr f \Bigr\},$$
where the infimum runs over non-negative Borel measures $\mu$ satisfying
$$(\mu\ast g)(x) = \int_{\mathbb{R}^n} g(x-t)\, d\mu(t) \gr f(x)\quad \text{for all} \ x\in\mathbb{R}^n.$$

Let $LC_{g}(\mathbb{R}^{n})$ denote the class of geometric log-concave functions; these are the upper 
semi-continuous log-concave functions $f:\mathbb{R}^{n}\to[0,\infty)$ with $f(0)=\|f\|_{\infty}=1$. If 
$f=e^{-\varphi}\in LC_{g}(\mathbb{R}^{n})$, its Legendre dual is $f^{\ast}=e^{-\mathcal L\varphi}$, where
$$\mathcal{L}\varphi(x)= \sup_{y\in\mathbb{R}^{n}}\bigl\{\langle x,y\rangle - \varphi(y)\bigr\}$$
is the Legendre transform of $\varphi$. In the functional setting, the dual function $f^{\ast}$ plays a role analogous to that 
of the polar body in classical convex geometry: many geometric inequalities relating 
a convex body to its polar admit functional counterparts involving a log-concave 
function and its Legendre dual. A notable example is the functional Blaschke--Santal\'o inequality,
\begin{equation*} \int_{\mathbb{R}^n} \exp (-\varphi(x))\,dx\cdot \int_{\mathbb{R}^n} \exp (-\mathcal{L}\varphi(x))\,dx 
\ls (2\pi)^n.\end{equation*}
The natural analogue of the Euclidean ball is the Gaussian
$$g(x)=\exp\bigl(-\tfrac{1}{2}|x|^{2}\bigr),\qquad\text{for which } g^{\ast}=g.$$
Artstein-Avidan and Slomka~\cite{Artstein-Slomka-2021} established the existence of a
functional version of Milman's $M$-position for geometric log-concave
functions.

\begin{theorem}[Artstein--Slomka] \label{th:M-covering}
Let $f:\mathbb{R}^{n}\to[0,\infty)$ be a geometric log-concave function.
There exists $T\in GL_{n}$ such that $\widetilde f=f\circ T$ satisfies
$\int \widetilde f = (2\pi)^{n/2}$ and
$$\max\Bigl\{N(\widetilde f,g),\, N(\widetilde f^{\ast},g),\, N(g,\widetilde f),\, N(g,\widetilde f^{\ast})
\Bigr\} \;\ls\; C^{n},$$
for an absolute constant $C>0$.
\end{theorem}

In this work we establish the existence of a regular functional $M$-position, in the sense of Pisier, for geometric log-concave 
functions. Pisier~\cite{Pisier-1989} constructed an entire family of $M$-positions for any symmetric convex body $K\subset\mathbb{R}^{n}$, 
providing quantitative control of covering numbers at all scales.

\begin{theorem}[Pisier] \label{th:pisier-alpha-regular}
Let $0<\alpha<2$ and let $K\subset\mathbb{R}^{n}$ be a symmetric convex body. Then $K$ has a linear image $\widetilde K$ such that
$$\max\Bigl\{ N(\widetilde K, t B_{2}^{n}),\ N(B_{2}^{n}, t\widetilde K),\ N(\widetilde K^{\circ}, t B_{2}^{n}),\
N(B_{2}^{n}, t\widetilde K^{\circ})\Bigr\}\;\ls\;\exp\!\left(c(\alpha)\, n/t^{\alpha}\right)$$
for every $t\gr c(\alpha)^{1/\alpha}$, where $c(\alpha)=O\!\bigl((2-\alpha)^{-\alpha/2}\bigr)$ as $\alpha\to 2^{-}$.
\end{theorem}

A convex body satisfying the above is said to be in $\alpha$-regular $M$-position.

\medskip

\textbf{Main results.} We prove that geometric log-concave functions admit an almost $1$-regular functional $M$-position.  
Moreover, we show that the isotropic position already enjoys this regularity. Background material on isotropic convex bodies 
and isotropic log-concave functions is provided in Section~\ref{section:2}. In our results, homothetic dilation of 
log-concave functions is given by
$$(t\odot f)(x)=f(x/t), \qquad t>0.$$

\begin{theorem}\label{th:regular-M-covering}
Let $f:\mathbb{R}^{n}\to[0,\infty)$ be an isotropic geometric log-concave function.
Then for every $t\gr 1$,
$$\max\Bigl\{N(f,\, t\odot g),\, N(g,\, t\odot f^{\ast})\Bigr\}
\;\ls\;\exp\!\left( \frac{\gamma_{n} n}{t} \right),$$
and 
$$\max\Bigl\{N(f^{\ast},\, t\odot g),\, N(g,\, t\odot f),\,\Bigr\}
\;\ls\;\exp\!\left( \frac{\delta_{n} n}{t} \right),$$
where $\gamma_{n} \ls c(\ln n)^{2}$ and $\delta_n\ls c\ln n$.
\end{theorem}

Thus the isotropic position yields a universal functional $M$-position whose
regularity exponent is arbitrarily close to~$1$.

\smallskip

We outline the main ideas of the proof.   
To each isotropic geometric log-concave function $f$ we associate the convex body
$$R_f:=\{x\in\mathbb{R}^n:f(x)\gr\exp(-50n)\}.$$
One may compare $R_f$ with the almost isotropic convex body $K_{n+1}(f)$ introduced by K.~Ball \cite{Ball-1988}, 
and verify that their geometric distance is bounded by an absolute constant.  
We then use an observation of the second named author \cite{Tziotziou-2025}, based on 
E.~Milman's sharp $M^{\ast}$-estimate \cite{EMilman-2014} and the recent optimal $M$-estimate of 
Bizeul and Klartag \cite{Bizeul-Klartag-2025} for isotropic convex bodies, to show that $R_f$ satisfies almost $2$-regular 
covering estimates.  
These geometric bounds are transferred to the functional covering numbers 
$N(f,\, t\odot g)$ and $N(g,\, t\odot f)$ using a decomposition of 
$\exp(-50n\|\cdot\|_{R_f})$, a corresponding decomposition of the Gaussian $g$, 
and basic structural properties of functional covering numbers.

A dual argument applies to the Legendre transform.  
A result of Fradelizi and Meyer \cite{Fradelizi-Meyer-2008} implies that
$$50n\,(R_f)^{\circ}\subseteq R_{f^{\ast}}\subseteq 100n\, (R_f)^{\circ}.$$
Combined with the Blaschke--Santal\'o and Bourgain--Milman inequalities, 
this yields analogous regularity for the covering numbers of $R_{f^{\ast}}$, and hence 
for $N(f^{\ast},\, t\odot g)$ and $N(g,\, t\odot f^{\ast})$.

\smallskip

We also show that isotropic geometric log-concave functions satisfy the conclusion of 
Theorem~\ref{th:M-covering}. That is, the isotropic position already provides a universal functional $M$-position in the sense of Milman.

\begin{theorem}\label{th:isotropic-milman}
Let $f:\mathbb{R}^{n}\to[0,\infty)$ be an isotropic geometric log-concave function.
Then,
$$\max\Bigl\{N(f,g),\, N(f^{\ast},g),\, N(g,f),\, N(g,f^{\ast})\Bigr\} \;\ls\; C^{n},$$
for an absolute constant $C>0$.
\end{theorem}

The proof combines ideas from the proof of Theorem~\ref{th:M-covering} 
in \cite{Artstein-Slomka-2021} with techniques employed in the proof of 
Theorem~\ref{th:regular-M-covering}.  Compared with Theorem~\ref{th:regular-M-covering}, 
Theorem~\ref{th:isotropic-milman} yields sharper estimates for the covering numbers 
in the regime where $t\gr 1$ is bounded above by a small power of $\ln n$.  
Thus, the isotropic position furnishes an efficient and robust functional 
$M$-position without requiring additional regularization assumptions.

\medskip 

Our second main result shows that similar regular estimates for the functional covering numbers hold true
for another choice of the dual of $f$, which is based on the polarity transform. 
The polar function $\varphi^{\circ}$ of a convex lower semi-continuous function $\varphi:\mathbb{R}^n\to [0,\infty]$
with $\varphi(0)=0$ is defined by the $\mathcal{A}$-transform of $\varphi$:
$$\varphi^{\circ}(x)=(\mathcal{A}\varphi)(x)=\sup_{y\in\mathbb{R}^n}\frac{\langle x,y\rangle -1}{\varphi(y)}.$$
The definition of the $\mathcal{A}$-transform appears in the book by Rockafellar \cite[page~136]{Rockafellar-book}, where it is 
also proved that it commutes with the Legendre transform. However, the polarity transform was introduced 
and studied in depth by Artstein-Avidan and Milman in \cite{Artstein-VMilman-2011} as the functional
extension of convex-body polarity and plays a central role in functional versions of the 
Blaschke--Santal\'{o} and Bourgain--Milman inequalities.

Consider the geometric log-concave function $f=e^{-\varphi}$. A result of V.~Milman and Rotem from \cite{VMilman-Rotem-2013}
implies that if we consider the scaled polar function
$$\varphi_{\mathcal{A}}(x)=(50n)^2\varphi^{\circ}(x/n)$$
and if we define $f_{\mathcal{A}}=e^{-\varphi_{\mathcal{A}}}$ then
$$n(R_f)^{\circ}\subseteq R_{f_{\mathcal{A}}}\subseteq 2n(R_f)^{\circ}.$$
Using the same strategy as in the proof of Theorem~\ref{th:regular-M-covering} we show that
$f_{\mathcal{A}}$ also admits regular covering estimates.

\begin{theorem}\label{th:regular-airy}
Let $f:\mathbb{R}^{n}\to[0,\infty)$ be an isotropic geometric log-concave function.
Then for every $t\gr 1$,
$$N(g,\, t\odot f_{\mathcal{A}})\;\ls\;\exp\!\left( \frac{\gamma_{n} n}{t} \right)
\quad \text{and} \quad  
N(f_{\mathcal{A}},\, t\odot g)\;\ls\;\exp\!\left( \frac{\delta_{n} n}{t} \right),$$
where $\gamma_{n} \ls c(\ln n)^{2}$ and $\delta_n\ls c\ln n$.
\end{theorem}

We also obtain the corresponding analogue of Theorem~\ref{th:isotropic-milman}.

\begin{theorem}\label{th:isotropic-milman-dual}
Let $f:\mathbb{R}^{n}\to[0,\infty)$ be an isotropic geometric log-concave function.
Then,
$$\max\Bigl\{N(f,g),\, N(f_{\mathcal{A}},g),\, N(g,f),\, N(g,f_{\mathcal{A}})\Bigr\} \;\ls\; C^{n},$$
for an absolute constant $C>0$.
\end{theorem}

Together, Theorems~\ref{th:regular-M-covering} and \ref{th:regular-airy} show that, in the isotropic position, 
a log-concave function $f$ and its duals $f^{\ast}$ and $f_{\mathcal{A}}$ behave like well-balanced 
Gaussians at all scales.  
This establishes a functional analogue of Pisier's theorem on the existence of regular $M$-positions for convex bodies, 
and may have further applications in the analysis of log-concave functions.

\smallskip

For background on isotropic convex bodies and log-concave measures and functions, see~\cite{BGVV-book};  
for general information on the local theory of normed spaces, see~\cite{AGA-book, AGA-book-2, Pisier-book}.

\section{Convex bodies and log-concave functions}\label{section:2}

We work in $\mathbb{R}^n$, equipped with the standard inner product $\langle \cdot, \cdot \rangle$.  
The corresponding Euclidean norm is denoted by $|\cdot|$, the Euclidean unit ball by $B_2^n$, and the Euclidean unit sphere by $S^{n-1}$.  
Volume in $\mathbb{R}^n$ is denoted by $\vol_n$, and we write $\omega_n = \vol_n(B_2^n)$ for the volume of the unit ball.
We denote by $\sigma$ the rotationally invariant probability measure on $S^{n-1}$.

\smallskip

Throughout the text, the symbols $c, c', c_1, c_2, \ldots$ denote absolute positive constants whose values may change from line to line.  
Whenever we write $a \approx b$, we mean that there exist absolute constants $c_1, c_2 > 0$ such that $c_1 a \ls b \ls c_2 a$. 
Similarly, for subsets $K,T\subseteq \mathbb{R}^n$, we write $K\approx T$ if
$c_1K\subseteq T\subseteq c_2K$ for some absolute constants $c_1,c_2>0$.

\bigskip 

\textbf{2.1. Convex bodies.} A convex body in $\mathbb{R}^n$ is a compact convex set $K$ with nonempty interior.  
It is called symmetric if $K = -K$, and centered if its barycenter 
$\operatorname{bar}(K) = \frac{1}{\vol_n(K)} \int_K x\,dx$ is at the origin. If $K$ and $T$ are two convex bodies in ${\mathbb R}^n$ that contain
the origin in their interior, their geometric distance $d_{{\rm G}}(K,T)$ is defined by
$$d_{{\rm G}}(K,T)=\inf\{ ab:a,b>0,K\subseteq bT\;\hbox{and}\;T\subseteq aK\}.$$

The radial function of a convex body $K$ with $0 \in \operatorname{int}(K)$ is defined by
$\rho_K(x) = \max\{ t > 0 : t x \in K \}$ for $x \neq 0$,
and the support function of $K$ is given by $h_K(y) = \max\{ \langle x, y \rangle : x \in K \}$, 
$y \in \mathbb{R}^n$.
The volume radius of $K$ is
$$\operatorname{vrad}(K)= r_K=\left( \frac{\vol_n(K)}{\vol_n(B_2^n)} \right)^{1/n}.$$
The polar body of a convex body $K$ with $0 \in \operatorname{int}(K)$ is defined as
$$K^\circ = \{ x \in \mathbb{R}^n : \langle x, y \rangle \ls 1 \text{ for all } y \in K \}.$$

The Blaschke--Santal\'{o} inequality states that if $K$ is a convex body in $\mathbb{R}^n$ such that either 
$\operatorname{bar}(K) = 0$ or $\operatorname{bar}(K^\circ) = 0$, then
$$\vol_n(K)\,\vol_n(K^\circ) \ls \omega_n^2.$$
In the opposite direction, the Bourgain--Milman inequality guarantees that if $K$ is a convex body in $\mathbb{R}^n$ 
with $0 \in \operatorname{int}(K)$, then
$$\vol_n(K)\,\vol_n(K^\circ)\gr c^n \omega_n^2,$$
where $c>0$ is an absolute constant. These classical results can be found, for example, in \cite{AGA-book}.

\bigskip 

\textbf{2.2. Log-concave functions.} A function $f : \mathbb{R}^n \to [0,\infty)$ is called log-concave if it can be written in the form  
$f = e^{-\varphi}$, where $\varphi : \mathbb{R}^n \to (-\infty,\infty]$ is a proper, lower semi-continuous (l.s.c.) convex
function. Properness means that $\mathrm{dom}(\varphi) := \{x \in \mathbb{R}^n : \varphi(x) < \infty\} \neq \varnothing$,
and l.s.c.\ convexity ensures that $f$ is upper semi-continuous and satisfies the classical log-concavity inequality  
$f\big((1-\lambda)x+\lambda y\big) \gr f(x)^{1-\lambda} f(y)^\lambda$ for all $x,y \in \mathbb{R}^n,\; \lambda \in (0,1)$.

We call $f=e^{-\varphi}$ a geometric log-concave function if it also satisfies the normalization
$$f(0)=\|f\|_{\infty}=1,$$
which is equivalent to requiring that the associated convex function $\varphi$ satisfies
$$\varphi : \mathbb{R}^n \to [0,\infty], \qquad \varphi(0)=0,$$
together with properness, convexity, and lower semi-continuity (these are the geometric convex functions).
We denote by $LC_g(\mathbb{R}^n)$ the class of all such geometric log-concave functions.

For any convex body $K\subset\mathbb{R}^n$, the indicator function $\mathds{1}_K$ is a geometric log-concave function.
Indeed, let $\mathds{1}_K^\infty$ be the convex indicator of $K$,  
$$
\mathds{1}_K^\infty(x) = 
    \begin{cases}
        0, & x\in K, \\
        \infty, & x\notin K,
\end{cases}
$$
which is a proper l.s.c.\ convex function with $\mathds{1}_K^\infty(0)=0$ whenever $0\in K$.  
Then $\mathds{1}_K = \exp(-\mathds{1}_K^\infty)$.

\smallskip

For $t>0$ and a log-concave function $f$, the functional homothety is defined by
$$(t \odot f)(x) = f\big(x/t\big),$$
which corresponds to replacing $\varphi$ by $\varphi_t(x)= \varphi(x/t)$.
This transformation respects log-concavity and is a natural functional counterpart of geometric dilation.

\smallskip

Let $\varphi:\mathbb{R}^n\to[0,\infty]$ be a geometric convex function. The Legendre transform of $\varphi$ is
$$\mathcal{L}\varphi(x)= \sup_{y\in\mathbb{R}^n}\{ \langle x,y\rangle - \varphi(y)\}.$$
It is always a convex, l.s.c.\ function, and satisfies the involution property 
$\mathcal{L}(\mathcal{L}\varphi)=\varphi$ if $\varphi$ is proper, l.s.c., convex. 
The Legendre dual of $f=e^{-\varphi}$ is then defined as
$$f^{\ast}(x)=\exp\big(-\mathcal{L}\varphi(x)\big).$$
The fundamental example of a self-dual log-concave function is the Gaussian
$$g(x) = \exp\!\left(-\tfrac{1}{2}|x|^{2}\right),$$
which satisfies  
$$\int_{\mathbb{R}^n} g(x)\,dx = (2\pi)^{n/2}, \qquad g^{\ast}=g.$$
Given two log-concave functions $f=e^{-\varphi}$ and $g=e^{-\psi}$, we define the sup-convolution or 
Asplund product of $f$ and $g$ by
$$(f\star g)(x)= \sup_{y\in\mathbb{R}^n} f(y)\, g(x-y)= \exp\!\big(-(\varphi \square \psi)(x)\big),$$
where the inf-convolution of two convex functions $\varphi$ and $\psi$ is
$$(\varphi \square \psi)(x)=\inf_{y\in\mathbb{R}^n}\big\{\varphi(y)+\psi(x-y)\big\}.$$
In view of the identity $\mathcal{L}(\varphi\square\psi)=\mathcal{L}\varphi +\mathcal{L}\psi$,
the operation $f\star g$ allows us to define the analogue of Minkowski addition for log-concave functions.

\smallskip 

The polar (or $\mathcal{A}$-transform) of $\varphi$ is defined by
$$\varphi^{\circ}(x)=(\mathcal{A}\varphi)(x)=\sup_{y\in\mathbb{R}^n}\frac{\langle x,y\rangle -1}{\varphi(y)}.$$
This transform is a functional analogue of the classical polarity of convex bodies.  
If $\varphi = \mathds{1}_K^\infty$, then $\varphi^\circ = \mathds{1}_{K^\circ}^\infty$ and  
$$e^{-\varphi} = \mathds{1}_K\quad\Longrightarrow\quad e^{-\varphi^\circ} = \mathds{1}_{K^\circ}.$$
The polar transform plays an essential role in functional 
analogues of the Blaschke--Santal\'{o} inequality. We refer to \cite[Chapter~9]{AGA-book-2} for 
more information and references.

We shall work with the scaled version  
$$\varphi_{\mathcal{A}}(x)=(50n)^2\, \varphi^{\circ}(x/n)$$
and define the polar log-concave function of $f=\exp(-\varphi)$ by
$$f_{\mathcal{A}}(x)=\exp\big(-\varphi_{\mathcal{A}}(x)\big).$$

We would like to mention here that Gilboa, Segal and Slomka \cite{Gilboa-Segal-Slomka-2025} 
have also used some scaled version of the polarity transform to study the Mahler product
of geometric log-concave functions. More precisely, they showed that if $q\approx n^2$ then
$$\left(\int_{\mathbb{R}^n}e^{-\varphi(x)}dx\right)^{1/n}\left(\int_{\mathbb{R}^n}e^{-q\mathcal{A}\varphi(x)}dx\right)^{1/n}
\approx\frac{1}{n}$$
for every centered geometric log-concave function $f=e^{-\varphi}$ with finite positive integral.
They also obtained an analogous result for the $\mathcal{J}$-transform, defined as 
$\mathcal{J}=\mathcal{L}\mathcal{A}=\mathcal{A}\mathcal{L}$.

\bigskip 

\textbf{2.3. Isotropic geometric log-concave functions.} Let $f:\mathbb{R}^n\to [0,\infty)$ be a log-concave function 
with finite positive integral. The barycenter of $f$ is defined by
\[
\operatorname{bar}(f)
= \frac{\int_{\mathbb{R}^n} x\, f(x)\, dx}{\int_{\mathbb{R}^n} f(x)\, dx},
\]
and its isotropic constant is the affine-invariant quantity
\begin{equation}\label{eq:definition-isotropic}
L_f:= \left( \frac{\|f\|_{\infty}}{\int_{\mathbb{R}^n} f(x)\, dx} \right)^{1/n}
\det(\operatorname{Cov}(f))^{1/(2n)},
\end{equation}
where $\operatorname{Cov}(f)$ is the covariance matrix of $f$ with entries
\begin{equation*}
\textrm{Cov}(f)_{i,j}:=\frac{\int_{\mathbb{R}^n}x_ix_j f(x)\,dx}{\int_{\mathbb{R}^n} f(x)\,dx}-\frac{\int_{\mathbb{R}^n}x_i f(x)\,dx}{\int_{\mathbb{R}^n} f(x)\,dx}\frac{\int_{\mathbb{R}^n}x_j f(x)\,dx}{\int_{\mathbb{R}^n} f(x)\,dx}.
\end{equation*} In this article, a log-concave function $f$ is called isotropic if
$$\operatorname{bar}(f)=0,\quad \int_{\mathbb{R}^n}f(x)\,dx=1\quad \text{and} \quad \operatorname{Cov}(f)=\lambda_f^2 I_n$$
for some $\lambda_f>0$.  We  say that $f$ is an isotropic geometric log-concave function if it also satisfies
$f(0)=\|f\|_{\infty}=1$.

A convex body $K$ in $\mathbb{R}^n$ is called isotropic if $\vol_n(K)=1$, ${\rm bar}(K)=0$, and
${\rm Cov}(\mu_K)=L_{\mu_K}^2I_n$, where $\mu_K$ is the uniform measure on $K$. Note that $K$ is isotropic 
if and only if its indicator function $\mathds{1}_K$ is an isotropic geometric log-concave function.

\smallskip

It is straightforward to check that any centered log-concave function $f$ admits an invertible linear map  
$T\in GL_n$ such that the covariance matrix of $f_1 := f\circ T$ is a multiple $\lambda_{f_1}^2 I_n$ of the identity; moreover, $L_{f_1}=L_f$.  
Dividing $f_1$ by $\|f_1\|_{\infty}$ yields a centered log-concave function $f_2$ with
$$\|f_2\|_{\infty}=1,\quad \int_{\mathbb{R}^n} f_2(x)\,dx = a^n,\quad a := \lambda_{f_1}/L_f.$$
Define $f_3(x) := f_2(a x)$. Then
$$\|f_3\|_{\infty}=1, \qquad \int_{\mathbb{R}^n} f_3(x)\, dx = 1, \qquad\operatorname{Cov}(f_3) = L_f^2 I_n.$$
A log-concave function satisfying these properties is called a normalized isotropic log-concave function.  
Thus every centered log-concave $f$ with finite positive integral admits a normalized isotropic position  
$\tilde f = f\circ T$ for some $T\in GL_n$. If we make the additional assumption that $f(0)=\|f\|_{\infty}$ then
we have that $\tilde{f}\in LC_g(\mathbb{R}^n)$, i.e. $\tilde{f}$ is an isotropic geometric log-concave function.

\smallskip

Bourgain's slicing problem \cite{Bourgain-1986} asks if there exists an absolute constant $C>0$ such that
\begin{equation}\label{eq:conjecture}L_n:=\max\{ L_{\mu_K}:K\ \hbox{is an isotropic convex body in}\ \mathbb{R}^n\}\ls C.\end{equation}
K.~Ball~\cite{Ball-1988} proved that for every $n$,
$$\sup_f L_f \;\ls \; C_1L_n,$$
where the supremum is taken over all isotropic log-concave functions $f$ on $\mathbb{R}^n$.
Bourgain~\cite{Bourgain-1991} showed that $L_n \ls c\, n^{1/4}\ln n$, improved by Klartag~\cite{Klartag-2006} to $L_n \ls c\, n^{1/4}$.  
These were the best known bounds until 2020. In a breakthrough, Chen~\cite{C} proved that for every $\epsilon>0$,
$$L_n \ls n^{\epsilon}\qquad \text{for all sufficiently large $n$}.$$
This initiated a series of developments culminating in the complete resolution of Bourgain's problem by Klartag and 
Lehec~\cite{KL}, who proved that $L_n\ls C$, building on an important contribution of Guan~\cite{Guan}. 
Shortly thereafter, Bizeul~\cite{Bizeul-2025} provided an alternative proof.

\bigskip 

\textbf{2.4. Covering numbers of isotropic convex bodies.} Let $K$ be a convex body in $\mathbb{R}^n$ with $0 \in \operatorname{int}(K)$.  
Define its Minkowski functional $\|x\|_K := \inf \{ t>0 : x \in tK \}$
and its support function $h_K(x) := \max \{ \langle x,y\rangle : y \in K \}$. Set
$$M(K) := \int_{S^{n-1}} \|x\|_K\, d\sigma(x), \qquad
M^{\ast}(K) := \int_{S^{n-1}} h_K(x)\, d\sigma(x),$$
where $\sigma$ is the rotationally invariant probability measure on $S^{n-1}$.

When $K$ is symmetric, the classical Sudakov and dual Sudakov inequalities~\cite[Chapter~4]{AGA-book} provide upper bounds 
on covering numbers in terms of $M(K)$ and $M^*(K)$:
\begin{equation}\label{eq:sudakov}N(K,\, t B_2^n)\;\ls\; \exp\!\left( c\, n \frac{M^{*}(K)^2}{t^2} \right),
\qquad
N(B_2^n,\, tK)\;\ls\; \exp\!\left( c\, n \frac{M(K)^2}{t^2} \right),\end{equation}
for every $t>0$, where $c>0$ is an absolute constant.  

E.~Milman~\cite{EMilman-2014} proved that if $K$ is isotropic then
\begin{equation}\label{eq:EM}
M^{\ast}(K)\;\ls\; C \sqrt{n}\, (\ln n)^2\, L_K\;\ls\; c_1 \sqrt{n}\, (\ln n)^2,
\end{equation}
the second inequality following from the boundedness of $L_n$.  
This dependence on $n$ is optimal up to logarithmic factors.

The dual estimate was recently obtained by Bizeul and Klartag~\cite{Bizeul-Klartag-2025}: if $K$ is isotropic, then
\begin{equation}\label{eq:BK}
M(K) \;\ls\; c_2 \frac{\log n}{\sqrt{n}}.
\end{equation}
Let $r_K$ denote the radius of the Euclidean ball with the same volume as $K$, i.e. 
$\vol_n(K)=\vol_n(r_K B_2^n)$. 
For isotropic bodies we have $\vol_n(K)=1$ and hence this radius depends only on the dimension:
$r_K=r_n=\omega_n^{-1/n}$, and $r_n \approx \sqrt{n}$. 

Combining the bounds \eqref{eq:EM} and \eqref{eq:BK} with \eqref{eq:sudakov}, the second named author~\cite{Tziotziou-2025} 
showed that every isotropic convex body is essentially in a $2$-regular $M$-position.  

\begin{proposition}\label{prop:M-position}
Let $K$ be an isotropic convex body in $\mathbb{R}^n$.  
Then for every $t\gr 1$,
\begin{align}
\max\{ N(K,\, t r_n B_2^n),\; N(B_2^n,\, t r_n K^{\circ}) \}
&\;\ls\;
\exp\!\left( \frac{\gamma_n^2\, n}{t^2} \right),
\label{eq:M-position-a}
\\[0.2cm]
\max\{ N(r_n B_2^n,\, t K),\; N(r_n K^{\circ},\, t B_2^n) \}
&\;\ls\;
\exp\!\left( \frac{\delta_n^2\, n}{t^2} \right),
\label{eq:M-position-b}
\end{align}
where $\gamma_n \ls c_1 (\ln n)^2$ and $\delta_n \ls c_2 \ln n$.
\end{proposition}

These estimates provide a quantitative control of covering numbers for isotropic convex bodies and their polars, which is crucial in applications.  
In particular, they will play an essential role in the proof of Theorem~\ref{th:regular-M-covering}.

\bigskip 

\textbf{2.5. Functional covering numbers.}
Recall that for any pair of functions $f,g \in LC_g(\mathbb{R}^n)$, the covering number of $f$ by $g$ is defined by
$$N(f,g)= \inf\{\mu(\mathbb{R}^n) : \mu \ast g \gr f\},$$
where the infimum is taken over all non-negative Borel measures $\mu$ on $\mathbb{R}^n$ satisfying
$$\int g(x-t)\, d\mu(t) \gr f(x),\qquad x \in \mathbb{R}^n.$$
Intuitively, $N(f,g)$ measures the minimal ``weight" of translates of $g$ needed to dominate $f$.  
This generalizes the classical notion of covering numbers for convex bodies: it is useful to note that 
if $K$ and $T$ are convex bodies in $\mathbb{R}^n$, then
\begin{equation} \label{eq:indicator-covering}
N(\mathds{1}_K,\mathds{1}_T) \ls N(K,T).
\end{equation}
For the proof of \eqref{eq:indicator-covering}, let $N = N(K,T)$ and choose points
$x_1,\ldots,x_N \in \mathbb{R}^n$ such that $K \subseteq \bigcup_{j=1}^N (x_j + T)$.
Let $\mu$ be the counting measure on $\{x_1,\ldots ,x_N\}$. Then
$$\int_{\mathbb{R}^n}\mathds{1}_T(x-t)\, d\mu(t)= \sum_{j=1}^N \mathds{1}_T(x-x_j)
= \sum_{j=1}^N \mathds{1}_{x_j+T}(x)\gr \mathds{1}_{\cup_{j=1}^N (x_j+T)}(x)\gr \mathds{1}_K(x),$$
which shows that $\mu \ast \mathds{1}_T \gr \mathds{1}_K$, and $\mu(\mathbb{R}^n)=N$.
Thus $N(\mathds{1}_K,\mathds{1}_T)\ls N(K,T)$.

\smallskip

Functional covering numbers satisfy several properties analogous to classical covering numbers.  
In the next lemma we collect the ones that will be used later in the proof of Theorem~\ref{th:regular-M-covering}.
A detailed proof of these properties can be found in \cite[Section~2]{Artstein-Slomka-2021}, 
and they hold more generally for non-negative measurable functions.

\begin{lemma}\label{lem:properties}
Let $f,g,h,f_i,g_i \in LC_g(\mathbb{R}^n)$. Then, for any $a,b>0$ and $T\in GL_n$,
\begin{enumerate}
\item[{\rm (i)}] $N(af,bg)=\dfrac{a}{b}N(f,g)$.
\item[{\rm (ii)}] $N(f\circ T,\, g\circ T)=N(f,g)$.
\item[{\rm (iii)}] $N(f_1+f_2,g)\ls N(f_1,g)+N(f_2,g)$.
\item[{\rm (iv)}] If $f_1 \ls f_2$ and $g_1 \gr g_2$, then $N(f_1,g_1)\ls N(f_2,g_2)$.
\item[{\rm (v)}] $N(f,g)\ls N(f,h)\,N(h,g)$.
\end{enumerate}
\end{lemma}

Artstein-Avidan and Slomka \cite{Artstein-Slomka-2021} also defined the separation number of $f$ by $g$:
$$M(f,g)= \sup\Big\{\int f\, d\rho : \rho\ast g \ls 1\Big\},$$
where the supremum is over all non-negative Borel measures $\rho$ on $\mathbb{R}^n$ that satisfy
$$\int g(x-t)\, d\rho(t)\ls 1\qquad \text{for all } x\in\mathbb{R}^n.$$
Any such measure $\rho$ is called a separation measure of $g$.
Recall that the separation number $M(K,T)$ of two convex bodies $K$ and $T$ in
$\mathbb{R}^n$ is the maximal cardinality of a $T$-separated subset of $K$, i.e.
$$M(K,T)= \max\Big\{ M\in\mathbb{N} :\exists\, x_1,\ldots,x_M \in K
\text{ such that }(x_i+T)\cap(x_j+T)=\varnothing\ \forall\, i\neq j\Big\}.$$
It is useful to note that
\begin{equation}\label{eq:indicator-separation}
M(K,T)\ls M(\mathds{1}_K,\mathds{1}_T).
\end{equation}
For the proof of \eqref{eq:indicator-separation}, let $M=M(K,T)$ and choose
$x_1,\ldots ,x_M\in K$ such that $(x_i+T)\cap(x_j+T)=\varnothing$ for all $1\ls i\neq j\ls M$.
Let $\mu$ be the counting measure on $\{x_1,\ldots,x_M\}$. Then
$$\int_{\mathbb{R}^n}\mathds{1}_T(x-t)\, d\mu(t)= \sum_{j=1}^M \mathds{1}_{x_j+T}(x)
= \mathds{1}_{\cup_{j=1}^M (x_j+T)}(x)\ls 1,$$
which shows that $\mu\ast\mathds{1}_T\ls 1$, and $\mu(\mathbb{R}^n)=M$.
Thus $M(K,T)\ls M(\mathds{1}_K,\mathds{1}_T)$.

\smallskip

A remarkable result of Artstein-Avidan and Slomka \cite{Artstein-Slomka-2021} shows that for log-concave functions, 
the notions of covering and separation essentially coincide (up to reflection):

\begin{theorem}[Artstein--Slomka]
Let $f, g\in LC_g(\mathbb{R}^n)$. Then, $M(f,\overline{g}) = N(f,g)$, where $\overline{g}(x)=g(-x)$.
\end{theorem}

In the same work, Artstein-Avidan and Slomka obtained the following general bounds. If   
$f,g$ are geometric log-concave functions, then
$$\frac{\int f^{2}(x)\,dx}{\|f\ast \overline{g}\|_{\infty}}
\;\ls\;N(f,g)\;\ls\;2^{n}\frac{\int f^{2}(x)\,dx}{\|f\ast g\|_{\infty}}.$$
If, in addition, $f$ and $g$ are even functions, then
$$\frac{\int f^{2}(x)\,dx}{\int f(x)g(x)\,dx}\;\ls\;N(f,g)\;\ls\;
2^{n}\frac{\int f^{2}(x)\,dx}{\int f(x)g(x)\,dx}.$$
Moreover, for every $p>1$,
$$\frac{\int f(x)\,dx}{\int g(x)\,dx}\;\ls\;N(f,g)\;\ls\;
\frac{\int (f\star \overline{g}^{\,p-1})(x)\,dx}{\int \overline{g}^{\,p}(x)\,dx}.$$
Using these inequalities, Artstein-Avidan and Slomka showed in \cite{Artstein-Slomka-2021} 
that if $f,g\in LC_g(\mathbb{R}^n)$ are centered, then
$$C^{-n}N(g^{\ast},f^{\ast})\ls N(f,g)\ls C^nN(g^{\ast},f^{\ast}).$$
where $C>0$ is an absolute constant.  
Later, Gilboa, Segal, and Slomka proved in \cite{Gilboa-Segal-Slomka-2025} that if 
$q\approx n^2$ and $\varphi,\psi$ are convex geometric functions such that either 
${\rm bar}(\varphi)=0$ or ${\rm bar}(\mathcal{A}\varphi)=0$, and likewise either 
${\rm bar}(\psi)=0$ or ${\rm bar}(\mathcal{A}\psi)=0$, then
$$C^{-n}N(e^{-q\mathcal{A}\psi},e^{-q\mathcal{A}\phi})\ls N(e^{-\varphi},e^{-\psi})\ls C^nN(e^{-q\mathcal{A}\psi},e^{-q\mathcal{A}\phi}).$$
These results may be viewed as functional analogues of the well-known inequality
$$C^{-n}N(T^{\circ},K^{\circ})\ls N(K,T)\ls C^nN(T^{\circ},K^{\circ})$$
due to K\"{o}nig and Milman (see \cite[Theorem~8.2.3]{AGA-book}), which holds for every pair 
of symmetric convex bodies $K$ and $T$ in $\mathbb{R}^n$.

\section{Regular functional covering numbers}\label{section:3}

In this section we prove that every centered geometric log-concave function admits a regular
covering $M$-position. Our approach relies on the existence of an almost $2$-regular $M$-position
for convex bodies (Proposition~\ref{prop:M-position}). In fact, in Theorem~\ref{th:regular-M-covering}
we show that every isotropic geometric log-concave function is in an almost $1$-regular $M$-position.

Let $f:\mathbb{R}^n \to [0,\infty)$ be a centered log-concave function with $f(0)>0$.  
We associate to $f$ two classical families of convex bodies, denoted by $\{R_t(f)\}_{t>0}$ and
$\{K_t(f)\}_{t>0}$. First, for every $t>0$, define
$$R_t(f)=\{x\in\mathbb{R}^n:\, f(x)\gr e^{-t}f(0)\}.$$
Since $f$ is log-concave, the sets $R_t(f)$ are convex, and clearly $0\in\mathrm{int}(R_t(f))$.
To show that $R_t(f)$ is bounded, recall that every log-concave function with finite positive
integral satisfies (see~\cite[Lemma~2.2.1]{BGVV-book}) the estimate
\begin{equation}\label{eq:A-B}
f(x)\ls A e^{-B|x|}
\qquad\text{for all } x\in\mathbb{R}^n,
\end{equation}
for some constants $A,B>0$.  
Thus, if $x\in R_t(f)$ then
$$|x|\ls \frac{1}{B}\bigl(\ln(A/f(0))+t\bigr).$$

The second family of convex bodies $K_t(f)$ was introduced by K.~Ball, who also proved their convexity
in~\cite{Ball-1988}. For every $t>0$, define
$$K_t(f)=\Bigl\{x\in\mathbb{R}^n:\, \int_0^\infty r^{t-1}f(rx)\,dr \gr \frac{f(0)}{t} \Bigr\}.$$
The radial function of $K_t(f)$ is given by
\begin{equation}\label{eq:radial-Kp}
\varrho_{K_t(f)}(x)=\left( \frac{1}{f(0)}\int_{0}^{\infty} t\, r^{t-1} f(rx)\,dr \right)^{1/t},\qquad x\neq 0.
\end{equation}
For $0<t\ls s$ one has the inclusions (see~\cite[Proposition~2.5.7]{BGVV-book})
\begin{equation}\label{eq:inclusions-Kp}
\frac{\Gamma(t+1)^{1/t}}{\Gamma(s+1)^{1/s}}\,K_s(f)\subseteq K_t(f)\subseteq
\left(\frac{\|f\|_{\infty}}{f(0)}\right)^{\frac{1}{t}-\frac{1}{s}} K_s(f).
\end{equation}
Moreover, since $f$ is assumed centered and log-concave, we have that
$\|f\|_{\infty}/f(0)\ls e^n$; this inequality is due to Fradelizi~\cite{Fradelizi-1997}.

The next relation between the bodies $K_t(f)$ and $R_t(f)$ follows directly from the definitions
(see~\cite[Proposition~2.3]{Giannopoulos-Tziotziou-2025}).

\begin{proposition}\label{prop:r-3}
Let $f:\mathbb{R}^n \to [0,\infty)$ be a centered log-concave function with $f(0)>0$.  For every $s\gr t>0$ we have
$$R_t(f)\subseteq e^{t/s} K_s(f).$$
\end{proposition}

In the opposite direction we use the following estimate
(see~\cite[Proposition~2.4]{Giannopoulos-Tziotziou-2025}).

\begin{proposition}\label{prop:r-4}
Let $f:\mathbb{R}^n \to [0,\infty)$ be a centered log-concave function with $f(0)>0$.  For every $t\gr 2n$,
$$R_{5t}(f)\supseteq \left(1-\frac{2n}{t}\right)K_t(f).$$
If in addition $f$ is even, then
$$R_{5t}(f)\supseteq \bigl(1-e^{-t}\bigr)K_t(f).$$
\end{proposition}

We now introduce the convex body
\begin{equation}\label{eq:def-Rf}R_f := R_{50n}(f).\end{equation}

\begin{lemma}\label{lem:bounded-gd}Let $f:\mathbb{R}^n \to [0,\infty)$ be a centered log-concave function with $f(0)>0$. 
There exists a centered convex body $K\subset\mathbb{R}^n$ such that
$$d_{{\rm G}}(R_f,K)\ls C,$$
where $C>0$ is an absolute constant. In fact, we may choose $K=K_{n+1}(f)$.
\end{lemma}

\begin{proof}Using Proposition~\ref{prop:r-3} and \eqref{eq:inclusions-Kp} we obtain
\begin{equation}\label{eq:R-K-1}
R_f = R_{50n}(f)\subseteq e\, K_{50n}(f)\subseteq \alpha_1 K_{n+1}(f),
\end{equation}
for some absolute constant $\alpha_1>0$.  
Using Proposition~\ref{prop:r-4} and \eqref{eq:inclusions-Kp}, and taking into account Fradelizi's inequality,
\begin{equation}\label{eq:R-K-2}
\frac{4}{5e} K_{n+1}(f)\subseteq \frac{4}{5} K_{10n}(f)=
\left(1-\frac{2n}{10n}\right)K_{10n}(f)\subseteq R_{50n}(f)= R_f.
\end{equation}
Thus,
\begin{equation}\label{eq:R-K}
\alpha_2 K_{n+1}(f)\subseteq R_f \subseteq \alpha_1 K_{n+1}(f),\qquad\alpha_2=\tfrac{4}{5e}.
\end{equation}
Since $K_{n+1}(f)$ is centered (see \cite[Proposition~2.5.3]{BGVV-book}),
the lemma is proved.\end{proof}

\begin{remark}\label{rem:almost}\rm If $f$ is isotropic, then the convex body $K_{n+1}(f)$ is almost isotropic in the sense of \cite[Definition~2.5.11]{BGVV-book}:
for any $T\in GL_n$ such that $T(K_{n+1}(f))$ is isotropic, we have that 
$$\alpha^{-1}B_2^n\subseteq T(B_2^n)\subseteq \alpha B_2^n,$$
where $\alpha \gr 1$ is an absolute constant. For a proof of this claim, see \cite[Proposition~2.5.12]{BGVV-book}. It follows that
$$N(K_{n+1}(f), tB_2^n)=N(T(K_{n+1}(f)), tT(B_2^n))\ls N(T(K_{n+1}(f)), \alpha^{-1}tB_2^n)$$
for every $t\gr 1$, and similar inequalities are valid for the other three covering numbers $N(B_2^n, tK_{n+1}(f))$, $N((K_{n+1}(f))^{\circ},tB_2^n)$ and 
$N(B_2^n, t(K_{n+1}(f))^{\circ})$. Therefore, \eqref{eq:M-position-a} and \eqref{eq:M-position-b}
hold true for $K_{n+1}(f)$.
\end{remark}

In what follows, we also set
\begin{equation}\label{eq:r_f}r_f=\left(\frac{\vol_n(R_f)}{\omega_n}\right)^{1/n}.\end{equation}
We need three intermediate results.

\begin{proposition}\label{prop:fcover-1}
Let $f:\mathbb{R}^n\to[0,\infty)$ be a geometric log-concave function.
If $\vol_n(R_f)^{1/n}\approx 1$ and $R_f$ satisfies 
\begin{equation}\max\Bigl\{N(R_f,\, t r_f B_2^n),\,
N(B_2^n,\, t r_f (R_f)^{\circ})\Bigr\}\ls \exp\left( \tfrac{c_1 \kappa_n^2 n}{t^2}\right),
\label{eq:R-covering-a}\end{equation}
for some $\kappa_n\gr 1$ and every $t\gr 1$, then
$$N(f,t\odot g)\ls \exp\!\left(\frac{\kappa_n n}{t}\right)\qquad\text{for all } t\gr 1.$$
\end{proposition}

\begin{proof}
We first claim that
\begin{equation}\label{eq:prop-fcover-1.1}
f(x)\ls \mathds{1}_{R_f}(x)+\exp\!\left(-50n\,\|x\|_{R_f}\right),\qquad x\in\mathbb{R}^n.
\end{equation}
If $x\in R_f$, then $f(x)\ls 1=\mathds{1}_{R_f}(x)$. If $x\notin R_f$, then $\lambda:=\|x\|_{R_f}>1$ and by definition
$f(x/\lambda)=\exp(-50n)f(0)$. Log-concavity yields
$$f(x/\lambda)\gr f(x)^{\frac{1}{\lambda}}f(0)^{1-\frac{1}{\lambda}}$$
and hence
\begin{equation}\label{eq:no-need-1}f(x)\ls \left(\frac{f(x/\lambda)}{f(0)}\right)^{\lambda}f(0)=\exp\!\left(-50n\,\|x\|_{R_f}\right)f(0)
\ls \exp\!\left(-50n\,\|x\|_{R_f}\right).\end{equation}
Next, decompose $\exp(-50n\|x\|_{R_f})$ as
\begin{equation}\label{eq:fcover-exp}
\exp(-50n\|x\|_{R_f})=\sum_{k=0}^{\infty}e^{-50n\|x\|_{R_f}}
\,\mathds{1}_{\{k\ls 50n\|x\|_{R_f}<k+1\}}(x)
\ls \sum_{k=0}^{\infty}e^{-k}\,\mathds{1}_{\frac{k+1}{50n}R_f}(x).
\end{equation}
Combining \eqref{eq:prop-fcover-1.1}, \eqref{eq:fcover-exp}, and Lemma~\ref{lem:properties}\,(iii),
\begin{equation}\label{eq:fcover-sum}
N(f,t\odot g)\ls N(\mathds{1}_{R_f},t\odot g)+\sum_{k=0}^{\infty}e^{-k}\,
N(\mathds{1}_{\frac{k+1}{50n}R_f},\,t\odot g).
\end{equation}

We begin with $N(\mathds{1}_{R_f},t\odot g)$.
By submultiplicativity,
\begin{equation}\label{eq:fcover-first}
N(\mathds{1}_{R_f},t\odot g)\ls N(\mathds{1}_{R_f},\,\mathds{1}_{\sqrt{\kappa_nt}r_f B_2^n})
\,N(\mathds{1}_{\sqrt{\kappa_nt}r_f B_2^n},\,t\odot g).
\end{equation}
By \eqref{eq:indicator-covering} and \eqref{eq:R-covering-a},
\begin{equation}\label{eq:fcover-first-1}
N(\mathds{1}_{R_f},\,\mathds{1}_{\sqrt{\kappa_nt}r_f B_2^n})\ls \exp\!\left(\frac{c\kappa_n n}{t}\right).
\end{equation}
Since $\vol(R_f)^{1/n}\approx 1$, we have $r_f\approx\sqrt{n}$, and for $x\in \sqrt{\kappa_nt}r_f B_2^n$,
$$\exp\!\left(-\frac{|x|^2}{2t^2}\right)\gr\exp\left(-\frac{\kappa_nr_f^2}{2t}\right) 
\gr\exp(-c\kappa_nn/t),$$
where $c>0$ is an absolute constant. Thus
$$\mathds{1}_{\sqrt{\kappa_nt}r_f B_2^n}(x)\ls e^{c\kappa_nn/t}(t\odot g)(x),$$
and hence (by Lemma~\ref{lem:properties}\,(i))
\begin{equation}\label{eq:fcover-first-2}
N(\mathds{1}_{\sqrt{\kappa_nt}r_f B_2^n},t\odot g)\ls e^{c\kappa_nn/t}.
\end{equation}
Combining \eqref{eq:fcover-first-1}--\eqref{eq:fcover-first-2} gives
\begin{equation}\label{eq:fcover-first-final}
N(\mathds{1}_{R_f},t\odot g)\ls \exp\!\left(\frac{c_1\kappa_n n}{t}\right).
\end{equation}

Next, we give an upper bound for the sum in \eqref{eq:fcover-sum}. For $k\gr 0$, using \eqref{eq:fcover-first-2}
and submultiplicativity, we write
\begin{align*}N(\mathds{1}_{\frac{k+1}{50n}R_f},t\odot g) &\ls N(\mathds{1}_{\frac{k+1}{50n}R_f},\mathds{1}_{\sqrt{\kappa_nt}r_f B_2^n})
\,N(\mathds{1}_{\sqrt{\kappa_nt}r_f B_2^n},t\odot g)\\
&\ls e^{c\kappa_nn/t} N(R_f,\sqrt{\kappa_nt}r_f B_2^n)\,N\!\left(B_2^n,\tfrac{50n}{k+1}B_2^n\right)\,.
\end{align*}
Using \eqref{eq:R-covering-a} and the bound $N(B_2^n,\lambda B_2^n)\ls (1+2/\lambda)^n$,
$$N(\mathds{1}_{\frac{k+1}{50n}R_f},t\odot g)\ls \exp\!\left(\frac{c_2\kappa_n n}{t}\right)e^{\frac{k+1}{25}}.$$
It follows that
$$\sum_{k=0}^\infty e^{-k}N(\mathds{1}_{\frac{k+1}{50n}R_f},t\odot g)\ls 
C_1\exp\!\left(\frac{c_2\kappa_nn}{t}\right),$$
and together with \eqref{eq:fcover-sum} and \eqref{eq:fcover-first-final},
$$N(f,t\odot g)\ls \exp\!\left(\frac{c_1\kappa_n n}{t}\right)
+C_1\exp\!\left(\frac{c_2\kappa_n n}{t}\right),$$
which gives the required bound.
\end{proof}

\begin{remark}\label{rem:fcover-1}\rm (i) In the proof of Proposition~\ref{prop:fcover-1} we used only the bound
$N(R_f,\, t r_f B_2^n)\ls \exp\left( \tfrac{c_1 \kappa_n^2 n}{t^2}\right)$ from \eqref{eq:R-covering-a}.
Moreover, in \eqref{eq:no-need-1} we only use the fact that $f(0)\ls\|f\|_{\infty}=1$. Therefore,
Proposition~\ref{prop:fcover-1} holds for every log-concave function $f:\mathbb{R}^n\to[0,\infty)$ with $\|f\|_{\infty}=1$
that satisfies the first bound in \eqref{eq:R-covering-a}.

\smallskip 

(ii) Note also that, instead of decomposing $\exp(-50n\|x\|_{R_f})$ into a sum, as in \eqref{eq:fcover-exp}, we may write
$$e^{-50n\|x\|_{R_f}}=50n\int_0^{\infty}e^{-50ns}\mathds{1}_{sR_f}(x,s)\,ds=\int_0^{\infty}e^{-u}\mathds{1}_{\frac{u}{50n}R_f}(x,u)\,du$$
and then estimate the functional covering number $N(\exp(-50n\,\|x\|_{R_f}),t\odot g)$ writing
$$N(\exp(-50n\,\|x\|_{R_f}),t\odot g)\ls \int_0^{\infty}e^{-u}N(\mathds{1}_{\frac{u}{50n}R_f},t\odot g)\,du.$$
Using the same argument as in the proof above, we get
$$N(\mathds{1}_{\frac{u}{50n}R_f},t\odot g)\ls \exp\!\left(\frac{c_2\kappa_n n}{t}\right)e^{\frac{u}{25}},$$
and hence
$$N(\exp(-50n\,\|x\|_{R_f}),t\odot g)\ls \exp\!\left(\frac{c_2\kappa_n n}{t}\right)\int_0^{\infty}e^{-u}e^{\frac{u}{25}}du.$$
The last integral is bounded by an absolute constant, and Proposition~\ref{prop:fcover-1} follows.
\end{remark}

\begin{proposition}\label{prop:fcover-2}
Let $f:\mathbb{R}^n\to[0,\infty)$ be a geometric log-concave function.
If $\vol_n(R_f)^{1/n}\approx 1$ and $R_f$ satisfies 
\begin{equation}\max\Bigl\{N(r_f B_2^n,\, t R_f),\,N(r_f R_f^{\circ},\, t B_2^n)\Bigr\}
\ls \exp\!\left( \tfrac{c_1 \kappa_n^2 n}{t^2} \right)\label{eq:R-covering-b}
\end{equation} 
for some $\kappa_n\gr 1$ and every $t\gr 1$, then
$$N(g,t\odot f)\ls \exp\!\left(\frac{\kappa_n n}{t}\right)\qquad\text{for all } t\gr 1.$$
\end{proposition}

\begin{proof}
Let $t\gr 1$. If $x\in R_f$, then by definition $f(x)\gr \exp(-50n)f(0)$, and log-concavity gives
\begin{equation}\label{eq:need-1}f(x/t)\gr f(x)^{1/t}f(0)^{1-1/t}=\left(\frac{f(x)}{f(0)}\right)^{1/t}f(0)\gr \exp(-50n/t)f(0)=\exp(-50n/t)\end{equation}
because $f(0)=1$. Hence
$$\exp(-50n/t)\,\mathds{1}_{R_f}(x)\ls (t\odot f)(x),$$
and Lemma~\ref{lem:properties}\,(iv) and (i) yield
\begin{equation}\label{eq:fcover-2.1}
N(g,t\odot f)\ls N(g,\exp(-50n/t)\,\mathds{1}_{R_f}) =e^{50n/t}\,N(g,\mathds{1}_{R_f}).
\end{equation}
Next decompose $g$ into spherical annuli:
\begin{equation}\label{eq:annuli}g(x)=\sum_{k=0}^{\infty}g(x)\,\mathds{1}_{\{ak r_f \ls |x|<a(k+1)r_f\}}(x)
\ls \sum_{k=0}^{\infty}e^{-a^2 k^2 r_f^2/2}\,\mathds{1}_{a(k+1)r_f B_2^n}(x),\end{equation}
for some $a>0$ to be determined. By Lemma~\ref{lem:properties}\,(iii),
\begin{equation}\label{eq:fcover-2.2}
N(g,\mathds{1}_{R_f})\ls \sum_{k=0}^{\infty}e^{-a^2 k^2 r_f^2/2}\,
N(\mathds{1}_{a(k+1)r_f B_2^n},\mathds{1}_{R_f}).
\end{equation}
For each $k\gr 0$,
\begin{align}\label{eq:lem-2-new}
N(\mathds{1}_{a(k+1)r_f B_2^n},\mathds{1}_{R_f})
&\ls N(a(k+1)r_f B_2^n,R_f)\\
\nonumber &\ls N\!\left(a(k+1)r_f B_2^n,\, \tfrac{1}{\sqrt{\kappa_nt}} r_f B_2^n\right)
\,N(r_f B_2^n,\, \sqrt{\kappa_nt} R_f).
\end{align}
Volumetric covering bounds give
$$N\!\left(a(k+1)r_f B_2^n,\,\tfrac{1}{\sqrt{\kappa_nt}} r_f B_2^n\right)\ls \bigl(2\sqrt{\kappa_nt}\,a(k+1)+1\bigr)^n,$$
and by \eqref{eq:R-covering-b},
$$N(r_f B_2^n,\sqrt{\kappa_nt}R_f)\ls \exp\!\left(\frac{c_1\kappa_nn}{t}\right).$$
Inserting the above into \eqref{eq:lem-2-new} we conclude that
$$N(\mathds{1}_{a(k+1)r_f B_2^n},\mathds{1}_{R_f})\ls \bigl(2\sqrt{\kappa_nt}\,a(k+1)+1\bigr)^n
\exp\!\left(\frac{c\kappa_n n}{t}\right),$$
Recall that $\vol_n(R_f)^{1/n}\approx 1$, which implies $r_f\approx\sqrt{n}$. Substituting into \eqref{eq:fcover-2.2} gives
$$N(g,\mathds{1}_{R_f})\ls \exp\!\left(\frac{c\kappa_n n}{t}\right)
\sum_{k=0}^{\infty}\bigl(e^{-c_0 a^2 k^2 + 2\sqrt{\kappa_nt}\,a(k+1)}\bigr)^{n}.$$
Choosing $a=c_2\sqrt{\kappa_nt}$ (for an absolute constant $c_2>0$ sufficiently large) makes the series bounded by an
absolute constant. Hence
$$N(g,\mathds{1}_{R_f})\ls C \exp\!\left(\frac{c\kappa_n n}{t}\right).$$
Inserting into \eqref{eq:fcover-2.1} yields the desired estimate.
\end{proof}

\begin{remark}\label{rem:fcover-2}\rm (i) In the proof of Proposition~\ref{prop:fcover-2} we used only the bound
$N(r_f B_2^n, t R_f)\ls \exp\left( \tfrac{c_1 \kappa_n^2 n}{t^2}\right)$ from \eqref{eq:R-covering-b}.
However, note that in \eqref{eq:need-1} we need to assume that $f(0)=\|f\|_{\infty}=1$, i.e. that $f$ is a geometric
log-concave function.

\smallskip 

(ii) Instead of decomposing $g(x)$ into a sum, as in \eqref{eq:annuli}, one may write
$$g(x)=\exp\left(-\tfrac{1}{2}|x|^2\right)=\int_0^{\infty}se^{-s^2/2}\mathds{1}_{sB_2^n}(x,s)\,ds$$
and then estimate the functional covering number $N(g,\mathds{1}_{R_f})$ writing
$$N(g,\mathds{1}_{R_f})\ls \int_0^{\infty}se^{-s^2/2}N(\mathds{1}_{sB_2^n},\mathds{1}_{R_f})\,ds.$$
Making the change of variables $s=aur_f$ and using the same argument as in the proof above, we get
$$N(\mathds{1}_{aur_f B_2^n},\mathds{1}_{R_f})\ls \bigl(2\sqrt{\kappa_nt}\,au+1\bigr)^n
\exp\!\left(\frac{c\kappa_n n}{t}\right),$$
and recalling that $r_f\approx\sqrt{n}$ we see that
$$N(g,\mathds{1}_{R_f})\ls \exp\!\left(\frac{c\kappa_n n}{t}\right)
\int_0^{\infty}\bigl(e^{-c_0 a^2u^2 + 3\sqrt{\kappa_nt}\,au}\bigr)^{n}du.$$
Choosing $a=c_2\sqrt{\kappa_nt}$ (for an absolute constant $c_2>0$ sufficiently large) makes 
the last integral bounded by an absolute constant, and Proposition~\ref{prop:fcover-2} follows.
\end{remark}

\begin{proposition}\label{prop:fcover-3}
Let $f:\mathbb{R}^n\to[0,\infty)$ be a centered geometric log-concave function such that  
$\vol_n(R_f)^{1/n}\approx 1$ and $R_f$ satisfies 
\begin{equation}N((R_f)^{\circ},\, t\,r^{\circ}B_2^n)\ls \exp\!\left(\frac{\delta_n^2 n}{t^2}\right) 
\quad \text{and}\quad  
N(r^{\circ}B_2^n,\, (R_f)^{\circ})\ls \exp\!\left(\frac{\gamma_n^2 n}{t^2}\right)\label{eq:R-covering-c-d}\end{equation}
for some $\gamma_n,\delta_n\gr 1$ and every $t\gr 1$, where $r^{\circ}B_2^n$ is the Euclidean ball having the same 
volume as $(R_f)^{\circ}$.  Define
$$R_{f^{\ast}}:= \{x\in\mathbb{R}^n:\, f^{\ast}(x)\gr e^{-50n}\},$$
where $f^{\ast}$ is the Legendre dual of $f$. Then $\vol_n(R_{f^{\ast}})^{1/n}\approx 1$ and 
$R_{f^{\ast}}$ satisfies
\begin{equation}\label{eq:R-covering-e-f}
N(R_{f^{\ast}},\, t r^{\ast} B_2^n)\ls \exp\!\left(\frac{c\delta_n^2 n}{t^2}\right)
\quad \text{and}\quad 
N(r^{\ast}B_2^n,\, t R_{f^{\ast}})\ls  \exp\!\left(\frac{c\gamma_n^2 n}{t^2}\right)
\end{equation}
for every $t\gr 1$, where $r^{\ast}B_2^n$ is the ball having the same volume as $R_{f^{\ast}}$
and $c>0$ is an absolute constant.
\end{proposition}

\begin{proof}
Write $f=e^{-\varphi}$, where $\varphi$ is a convex geometric function.  
Then $f^{\ast}=e^{-\mathcal{L}\varphi}$, where $\mathcal{L}\varphi$ is the Legendre transform of $\varphi$.
By \cite[Lemma~8]{Fradelizi-Meyer-2008}, for all $s,t>0$,
\begin{equation}\label{eq:fradelizi-meyer}
t\{x:\varphi(x)\ls t\}^{\circ}\;\subseteq\;\{y:\mathcal{L}\varphi(y)\ls t\}
\;\subseteq\;(t+s)\{x:\varphi(x)\ls s\}^{\circ}.
\end{equation}
Setting $s=t=50n$ in \eqref{eq:fradelizi-meyer} gives
\begin{equation}\label{eq:star=n-polar}50n\,(R_f)^{\circ}\subseteq R_{f^{\ast}}\subseteq 100n\, (R_f)^{\circ},\end{equation}
which implies
$$50n\,r^{\circ} \ls r^{\ast} \ls 100n\,r^{\circ}.$$
Using \eqref{eq:R-covering-c-d} for $(R_f)^{\circ}$, we obtain
$$N(R_{f^{\ast}},\, t r^{\ast} B_2^n)\ls N(100n\,(R_f)^{\circ},\, 50nt\,r^{\circ}B_2^n)
\ls \exp\!\left(\frac{c\delta_n^2 n}{t^2}\right),$$
and similarly,
$$N(r^{\ast}B_2^n,\, t R_{f^{\ast}})\ls N(100n\,r^{\circ}B_2^n,\, 50nt\, (R_f)^{\circ})
\ls \exp\!\left(\frac{c\gamma_n^2 n}{t^2}\right).$$

From Lemma~\ref{lem:bounded-gd} we know that $R_f$ has bounded geometric distance from a centered
convex body. Therefore, we may apply the Blaschke--Santal\'{o} and Bourgain--Milman inequalities 
(up to an absolute constant) to $R_f$. Combining with \eqref{eq:star=n-polar} yields
$$\vol_n(R_{f^{\ast}})^{1/n}\approx n\,\vol_n((R_f)^{\circ})^{1/n}\approx \vol_n(R_f)^{-1/n}\approx 1,$$
completing the proof.
\end{proof}

\begin{remark}\label{rem:fradelizi-meyer}\rm The precise statement of \cite[Lemma~8]{Fradelizi-Meyer-2008} is that if $\varphi:\mathbb{R}^n\to (-\infty,\infty]$
is a convex function then, for all $s,t>0$,
\begin{equation}\label{eq:fradelizi-meyer-2}
t\{x:\varphi(x)\ls t+\varphi(0)\}^{\circ}\;\subseteq\;\{y:\mathcal{L}\varphi(y)\ls t-\min\varphi\}\quad
\text{and}\quad 
\{y:\mathcal{L}\varphi(y)\ls t\}\;\subseteq\;(t+s)\{x:\varphi(x)\ls s\}^{\circ}.
\end{equation}
Assuming that $f$ is a centered log-concave function with $\|f\|_{\infty}=1$, we have $\min\varphi =0$ and $e^{-\varphi(0)}=f(0)\gr e^{-n}\|f\|_{\infty}=e^{-n}$
(by Fradelizi's inequality) which implies $\varphi(0)\ls n$. Setting $s=t=50n$ in \eqref{eq:fradelizi-meyer-2} we get
$$50n(\tilde{R}_f)^{\circ}\subseteq R_{f^{\ast}}\subseteq 100n(R_f)^{\circ},$$
where $\tilde{R}_f=R_{51n}(f)$. Log-concavity of $f$ implies that $\tilde{R}_f\subseteq \frac{51}{50}R_f$, and hence 
$$R_{f^{\ast}}\approx n(R_f)^{\circ}.$$
Then, we can follow the above argument and conclude that Proposition~\ref{prop:fcover-3} holds true for every
centered log-concave function $f:\mathbb{R}^n\to[0,\infty)$ for which $\|f\|_{\infty}=1$,  $\vol_n(R_f)^{1/n}\approx 1$ and $R_f$ satisfies 
\eqref{eq:R-covering-c-d} (the assumption that $f(0)=\|f\|_{\infty}$ is not needed).
\end{remark}

We are now ready for the proof of the main theorem.

\begin{proof}[Proof of Theorem~$\ref{th:regular-M-covering}$]
Let $f:\mathbb{R}^n\to[0,\infty)$ be an isotropic geometric log-concave function.  
Then
$$f(0)=\|f\|_{\infty}=1,\qquad\int_{\mathbb{R}^n} f(x)\,dx = 1.$$
The convex body $K_{n+1}(f)$ is almost isotropic, and Remark~\ref{rem:almost} shows that
\begin{align}
\max\!\bigl\{ N(K_{n+1}(f),\, t r_n B_2^n),\,N(B_2^n,\, t r_n (K_{n+1}(f))^{\circ}) \bigr\}
&\ls \exp\!\left( \frac{\gamma_n^2 n}{t^2} \right),
\label{eq:Kn+1-a}
\\[0.3em]
\max\!\bigl\{ N(r_n B_2^n,\, t K_{n+1}(f)),\,N(r_n (K_{n+1}(f))^{\circ},\, t B_2^n) \bigr\}
&\ls \exp\!\left( \frac{\delta_n^2 n}{t^2} \right),
\label{eq:Kn+1-b}
\end{align}
for all $t\gr 1$, where $\gamma_{n} \ls c(\ln n)^{2}$ and $\delta_n\ls c\ln n$.

Combining \eqref{eq:R-K} with \eqref{eq:Kn+1-a} and \eqref{eq:Kn+1-b}, we deduce that  
$R_f$ satisfies \eqref{eq:R-covering-a} with $\kappa_n=\gamma_n$ and \eqref{eq:R-covering-b} with $\kappa_n=\delta_n$.  
Moreover,
$$\vol_n(R_f)^{1/n}\approx \vol_n(K_{n+1}(f))^{1/n}\approx 1.$$
Applying Propositions~\ref{prop:fcover-1} and \ref{prop:fcover-2}, we obtain
$$N(f,\, t\odot g)\ls \exp\!\left(\frac{c\gamma_nn}{t}\right)\quad \text{and} \ N(g,\, t\odot f)
\ls \exp\!\left(\frac{c\delta_n n}{t}\right)
\qquad\text{for all } t\gr 1.$$
Next, Proposition~\ref{prop:fcover-3} shows that $R_{f^{\ast}}$ satisfies  
\eqref{eq:R-covering-c-d}, and moreover
$$\vol_n(R_{f^{\ast}})^{1/n}\approx 1.$$
Applying Propositions~\ref{prop:fcover-1} and \ref{prop:fcover-2} to $f^{\ast}$ yields
$$N(f^{\ast},\, t\odot g)\ls \exp\!\left(\frac{c\delta_n n}{t}\right)\quad \text{and}\quad  N(g,\, t\odot f^{\ast})
\ls \exp\!\left(\frac{c\gamma_n n}{t}\right)
\qquad\text{for all } t\gr 1.$$
This completes the proof.
\end{proof}

For the proof of Theorem~\ref{th:regular-airy} we shall use the next result of V.~Milman and Rotem (see \cite[Proposition~11]{VMilman-Rotem-2013}).

\begin{lemma}\label{lem:rotem-1}Let $\varphi:\mathbb{R}^n\to [0,\infty)$ be a geometric convex function.
For every $t>0$,
$$(\{x:\varphi(x)<1/t\})^{\circ}\subseteq \{x:\varphi^{\circ}(x)\ls t\}
\subseteq 2\,(\{x:\varphi(x)<1/t\})^{\circ}.$$
\end{lemma}

Consider the geometric log-concave function $f=e^{-\varphi}$. Applying Lemma~\ref{lem:rotem-1} with
$t=\frac{1}{50n}$ we get
\begin{equation}\label{eq:airy-1}(R_f)^{\circ}\subseteq \{x:\varphi^{\circ}(x)\ls 1/(50n)\}\subseteq 2(R_f)^{\circ}.\end{equation}
We define the scaled polar $\varphi_{\mathcal{A}}$ of $\varphi$ by 
$$\varphi_{\mathcal{A}}(x)=(50n)^2\varphi^{\circ}(x/n).$$
Note that $\varphi_{\mathcal{A}}(x)\ls 50n$ if and only if $\varphi^{\circ}(x/n)\ls 1/(50n)$. This shows that
\begin{equation}\label{eq:airy-2}n(R_f)^{\circ}\subseteq\{x:\varphi_{\mathcal{A}}(x)\ls 50n\}\subseteq 2n(R_f)^{\circ},\end{equation}
and hence, if we define $f_{\mathcal{A}}=e^{-\varphi_{\mathcal{A}}}$ we get
\begin{equation}\label{eq:airy-3}n(R_f)^{\circ}\subseteq R_{f_{\mathcal{A}}}\subseteq 2n(R_f)^{\circ}.\end{equation}

\begin{proof}[Proof of Theorem~$\ref{th:regular-airy}$]
We start as in the proof of Theorem~\ref{th:regular-M-covering}. Recall that
$$\vol_n(R_f)^{1/n}\approx \vol_n(K_{n+1}(f))^{1/n}\approx 1$$ 
and $R_f$ satisfies \eqref{eq:R-covering-a} with $\kappa_n=\gamma_n$ and \eqref{eq:R-covering-b} with $\kappa_n=\delta_n$.

\smallskip 

Combining \eqref{eq:airy-3} with the Blaschke--Santal\'{o} and Bourgain--Milman inequalities we get
\begin{equation*}
\vol_n(R_{f_{\mathcal{A}}})^{1/n}\approx 
n\,\vol_n((R_f)^{\circ})^{1/n}\approx
1 \big/ \vol_n(R_f)^{1/n}\approx 1.
\end{equation*}
In particular, if $r_{\mathcal{A}}$ denotes the radius of the ball that has volume $\vol_n(R_{f_{\mathcal{A}}})$,
we see that $r_{\mathcal{A}}\approx \sqrt{n}$.

Using \eqref{eq:R-covering-c-d} for $(R_f)^{\circ}$, we obtain
$$N(R_{f_{\mathcal{A}}},\, t r_{\mathcal{A}} B_2^n)\ls N(2n\,(R_f)^{\circ},\, tn\,r^{\circ}B_2^n)
\ls \exp\!\left(\frac{c\delta_n^2 n}{t^2}\right),$$
and similarly,
$$N(r_{\mathcal{A}}B_2^n,\, t R_{f_{\mathcal{A}}})\ls N(2n\,r^{\circ}B_2^n,\, tn\, (R_f)^{\circ})
\ls \exp\!\left(\frac{c\gamma_n^2 n}{t^2}\right).$$
Since $R_{f_{\mathcal{A}}}$ satisfies these bounds, and moreover
$$\vol_n(R_{f_{\mathcal{A}}})^{1/n}\approx 1,$$
applying Proposition~\ref{prop:fcover-1} with $\kappa_n=\delta_n$ and Proposition~\ref{prop:fcover-2} with $\kappa_n=\gamma_n$ to $f_{\mathcal{A}}$ yields
$$N(f_{\mathcal{A}},\, t\odot g)\ls \exp\!\left(\frac{c\delta_n n}{t}\right)\quad \text{and}\quad  N(g,\, t\odot f_{\mathcal{A}})
\ls \exp\!\left(\frac{c\gamma_n n}{t}\right)
\qquad\text{for all } t\gr 1.$$
This completes the proof.
\end{proof}

We now turn to the proofs of Theorems~\ref{th:isotropic-milman} and 
\ref{th:isotropic-milman-dual}. We shall use the fact that if $K$ is an 
isotropic convex body in $\mathbb{R}^n$, then
\begin{equation}\label{eq:survey}\max\big\{N(K,r_nB_2^n), N(r_nB_2^n,K), N( r_nK^{\circ},B_2^n),
N(B_2^n,r_nK^{\circ})\big\}\ls C^n\end{equation}
for some absolute constant $C>0$. This is a well-known consequence of the 
fact that $L_n \ls C$; see, for example, 
\cite[Theorem~3.3]{Giannopoulos-Pafis-Tziotziou-survey}.

\begin{proof}[Proof of Theorem~$\ref{th:isotropic-milman}$]
We begin as in the proof of Theorem~\ref{th:regular-M-covering}. 
Recall that $K_{n+1}(f)$ is almost isotropic (see Remark~\ref{rem:almost}), and hence satisfies 
\eqref{eq:survey}. Since $R_f$ is at bounded geometric distance from 
$K_{n+1}(f)$, we obtain $r_f \approx \sqrt{n}$, or equivalently,
$$\vol_n(R_f)^{1/n}\approx \vol_n(K_{n+1}(f))^{1/n}\approx 1,$$ 
and $R_f$ satisfies
\begin{equation}\label{eq:survey-Rf}\max\big\{N(R_f,r_fB_2^n), N(r_fB_2^n,R_f), N(r_n(R_f)^{\circ},B_2^n),
N(B_2^n,r_f(R_f)^{\circ})\big\}\ls C^n\end{equation}
for some absolute constant $C>0$. Following the proof of Proposition~\ref{prop:fcover-1} with $t=1$ and
\eqref{eq:R-covering-a} replaced by \eqref{eq:survey-Rf}, we deduce that
\begin{equation}\label{eq:M1}N(f,g)\ls C_1^n\end{equation}
for some absolute constant $C_1>0$.

Similarly, applying the proof of Proposition~\ref{prop:fcover-2} with $t=1$ 
and again replacing \eqref{eq:R-covering-a} by \eqref{eq:survey-Rf}, we find
\begin{equation}\label{eq:M2}N(g,f)\ls C_2^n\end{equation}
for some absolute constant $C_2>0$.

Next, recall from the proof of Proposition~\ref{prop:fcover-3} that
$$R_{f^{\ast}}\approx n\, (R_f)^{\circ},$$
and hence $r^{\ast} \approx n\, r^{\circ}$, where $r^{\ast}$ denotes the volume 
radius of $R_{f^{\ast}}$. In particular, $\vol_n(R_{f^{\ast}})^{1/n} \approx 1$. 
Since $(R_f)^{\circ}$ satisfies \eqref{eq:survey-Rf}, it follows that
$$N(R_{f^{\ast}},\, r^{\ast} B_2^n)\ls N(c_1n\,(R_f)^{\circ},\, c_2n\,r^{\circ}B_2^n)
\ls C_3^n,$$
and similarly,
$$N(r^{\ast}B_2^n,\, R_{f^{\ast}})\ls N(c_3n\,r^{\circ}B_2^n,\, c_4n\, (R_f)^{\circ})\ls C_4^n.$$
Repeating the proofs of \eqref{eq:M1} and \eqref{eq:M2} with $f^{\ast}$ in place 
of $f$ now yields
\begin{equation}\label{eq:M3-M4}N(f^{\ast},\, g)\ls C_5^n\quad \text{and}\quad  N(g,\,f^{\ast})
\ls C_6^n.\end{equation}
This completes the proof.\end{proof}

\begin{proof}[Proof of Theorem~$\ref{th:isotropic-milman-dual}$]We only need to prove that
\begin{equation}\label{eq:M5-M6}N(f_{\mathcal{A}},\, g)\ls C_7^n\quad \text{and}\quad  N(g,\,f_{\mathcal{A}})
\ls C_8^n.\end{equation}
The proof is similar to the second part of the proof of Theorem~\ref{th:isotropic-milman}, once we
recall that $$R_{f_{\mathcal{A}}}\approx n(R_f)^{\circ}$$
by \eqref{eq:airy-3}.
\end{proof}

\bigskip

\noindent {\bf Acknowledgement.} We are grateful to the referee for a careful and thoughtful reading of the manuscript. Their insightful comments and 
constructive suggestions have helped us weaken the assumptions and improve the estimates in some of the results. The second named author acknowledges support by a PhD scholarship
from the National Technical University of Athens.

\bigskip

\footnotesize
\bibliographystyle{amsplain}

\bigskip

\thanks{\noindent {\bf Keywords:} Covering numbers, log-concave functions, $M$-position, functional inequalities,
isotropic position, $MM^{\ast}$-estimate.

\smallskip

\thanks{\noindent {\bf 2020 MSC:} Primary 52A23; Secondary 46B06, 52A40, 52C17, 26B25.}

\bigskip

\bigskip

\noindent \textsc{Apostolos \ Giannopoulos}: School of Applied Mathematical and Physical
Sciences, National Technical University of Athens, Department of Mathematics, Zografou Campus, GR-157 80, Athens, Greece.

\smallskip

\noindent \textit{E-mail:} \texttt{apgiannop@math.ntua.gr}

\bigskip

\noindent \textsc{Natalia \ Tziotziou}: School of Applied Mathematical and Physical Sciences, National Technical University of Athens, Department of Mathematics, Zografou Campus, GR-157 80, Athens, Greece.

\smallskip

\noindent \textit{E-mail:} \texttt{nataliatz99@gmail.com}

\end{document}